\documentclass[12pt,letterpaper]{article}

\usepackage{amsfonts}
\usepackage{amsmath}\numberwithin{equation}{section}
\usepackage{amsthm}
\usepackage{amssymb}
\usepackage{bm}
\usepackage[tmargin=1.0in, lmargin=0.75in, rmargin=0.75in, bmargin=1.0in]{geometry}
\usepackage{hyperref}
\usepackage{mathrsfs}
\usepackage{mathtools}
\usepackage[T1]{fontenc}
\usepackage{xparse}
\usepackage{graphicx}
\usepackage[dvipsnames]{xcolor}
\usepackage{multirow}
\usepackage[all]{xy}

\newcommand{\R}{\mathbb{R}}
\newcommand{\Z}{\mathbb{Z}}
\newcommand{\A}{\mathscr{A}}

\let\oldd\d
\renewcommand{\d}{%
  \relax\ifmmode\mathrm{d}\else\oldd\fi
}
\renewcommand{\dh}{%
  \relax\ifmmode\mathrm{d}_{1,0}\else\olddh\fi
}
\newcommand{\dv}{\mathrm{d}_{0,1}}
\newcommand{\dR}{\mathrm{d}_{2,-1}}

\DeclareMathOperator{\ii}{\mathbf{i}}
\newcommand{\im}{\operatorname{im}}
\newcommand{\e}{\mathbf{e}}
\newcommand{\Lie}{\mathcal{L}}
\NewDocumentCommand{\Cinf}{G{}}{\IfNoValueTF{#1}{C^{\infty}}{C^{\infty}_{#1}}}
\DeclareMathOperator{\dcd}{\,\cdot\,,\cdot\,}
\DeclareMathOperator{\brk}{[\dcd]}

\renewcommand{\hat}{\widehat}
\renewcommand{\H}[1]{%
  \ifmmode\mathscr{H}^{#1}\else\oldH{#1}\fi
}
\let\oldH\H

\newtheorem{theorem}{Theorem}[section]

\newtheorem{corollary}[theorem]{Corollary}

\newtheorem{definition}[theorem]{Definition}
\newtheorem{example}[theorem]{Example}

\newtheorem{lemma}[theorem]{Lemma}

\newtheorem{remark}[theorem]{Remark}

\newtheorem{proposition}[theorem]{Proposition}

\numberwithin{equation}{section}
\begin{document}

\title{Hamiltonian semisprays on Lie algebroids}

\author{Misael Avenda\~{n}o Camacho, Jhonny Kama Mamani, Eduardo Velasco Barreras}

\maketitle

\begin{abstract}
We study the existence of Hamiltonian semisprays on Lie algebroids. This work is motivated by a problem studied by Vaisman for tangent bundles, and we extend this question to the setting of arbitrary Lie algebroids and provide a general solution.

More precisely, given a Lie algebroid and a regular Lagrangian, we construct a
family of Poisson brackets on the algebroid such that the Hamiltonian vector
field associated with the corresponding energy function is a semispray. Our approach
is based on the symplectic geometry of the prolongation of a Lie algebroid and a
cohomological analysis of its vertical subbundle. The results provide a geometric framework for second-order Hamiltonian dynamics
on Lie algebroids, extending  some known facts in the classical tangent bundle case and revealing new
interactions between Poisson geometry and algebroid structures.


\end{abstract}

\section{ Introduction}

The relationship between Hamiltonian dynamics and second-order differential equations is a classical topic in geometric mechanics, \cite{arnold1989mathematical,marsden1994introduction,marsden1984semidirect,szilasi2013connections}. 
In particular, the problem of determining whether Hamiltonian vector fields can be of second order has been studied in the context of tangent bundles. 
In the seminal work \cite{Vais1995}, Vaisman showed that, in general, Poisson structures on the tangent bundle may not admit second-order differential equations as Hamiltonian, and provided both counterexamples and specific constructions for which this property holds.

This paper aims to extend this problem to the broader setting of Lie algebroids. Given that Lie algebroids provide a natural framework for generalizing geometric mechanics beyond tangent bundles \cite{BarberoEtAl2016,CortesEtAl2009,CortesEtAl2006,dediego-marrero-2006,deLeonEtAl2009,deLeonEtAl2005,Grabowska-2006,IglesiasEtAl2007,Martinez2001,Weinstein1996}, it is natural to ask whether Vaisman’s problem admits a solution in this setting. More precisely, we address the following question: given a Lie algebroid, does there exist a Poisson structure admitting semisprays as Hamiltonian vector fields?

While Lagrangian and Hamiltonian formalisms on Lie algebroids are well established \cite{dediego-marrero-2006,deLeonEtAl2005,Martinez2001}, the existence of Poisson brackets with this specific dynamical property has not, to our knowledge, been addressed in this generality. The extension of Vaisman's results from the tangent bundle to general Lie algebroids is nontrivial,
since the tangent bundle of a Lie algebroid lacks several canonical structures present in a second tangent bundle,
and a general anchor and bracket introduce further difficulties.

Our main result shows that, for any regular Lagrangian on a Lie algebroid, one can construct a family of Poisson brackets such that the Hamiltonian vector field of the associated energy function is a semispray (Theorem \ref{maintheorem}). 

A key novelty of our approach lies in the use of the prolongation of a Lie algebroid \cite{carinena2007reduction,deLeonEtAl2005,martinez2002lie,Martinez2001}, 
which allows us to overcome intrinsic obstructions present on general Lie algebroids, such as the absence of symplectic structures.
This constitutes the main difference between the work of Vaisman on the tangent bundle and our approach to the same problem on Lie algebroids.
In his paper, Vaisman shows that a certain subfamily of the $\mathbb{V}$-Lagrangian symplectic forms admits Hamiltonian second-order differential equations. 

However, the arguments used by Vaisman cannot be directly extended to a general Lie algebroid. 
Indeed, dimension considerations already impose restrictions: if the total space of a Lie algebroid has odd dimension, then it does not admit any symplectic structure. Already in the case of even dimension, 
there is no canonical construction of a symplectic structure with a Hamiltonian semispray, in contrast with the tangent bundle case, where such a construction arises naturally.
For these reasons, we work on the prolongation of a Lie algebroid. In this setting, regular Lagrangians induce symplectic structures vanishing on the vertical subbundle such that its energy has for Hamiltonian section a second-order differential equation. In this paper, we enlarge and characterize the class of symplectic sections on the prolongation of a Lie algebroid in Propositions \ref{prop1} and \ref{prop2}, analogously to what Vaisman does in the tangent bundle. The final step in our construction is to prove that these symplectic sections on the prolongations induce Poisson brackets on the Lie algebroid with the properties established in Theorem \ref{maintheorem}.

A central ingredient in our analysis is a cohomological study of vertical subbundles
arising from the prolongation. In particular, we prove that the cohomology associated
with the vertical component of the Lie algebroid differential vanishes in positive degrees.
This result is crucial for constructing global primitives of certain differential forms
and for obtaining a structural description of symplectic sections adapted to the vertical
 subbundle. Although similar vanishing results can be obtained locally in coordinates, as in the
classical case considered by Vaisman, our approach provides a global and intrinsic
construction of the corresponding homotopy operator (Section \ref{sec4}). This is achieved by exploiting
the geometry of pullback bundles and representations of Lie algebroids, leading to a
coordinate-free proof that may be of independent interest.

As a consequence of this cohomological analysis, we obtain a characterization of a family
of symplectic sections on the prolongation that vanish on vertical directions. This
description allows us to construct a family of Poisson brackets explicitly on the Lie
algebroid for which some Hamiltonian dynamics are described by semisprays on the Lie algebroid.

The paper is organized as follows. In Section \ref{sec:prelim}, we recall the basic notions on Lie
algebroids and semisprays. Section \ref{sec:main} contains the statement of the main result.
Section \ref{sec4} is devoted to the cohomological analysis of vertical subbundles and the
construction of a homotopy operator, which plays a fundamental role in the sequel.
In Section \ref{sec:Prol}, we review the prolongation of a Lie algebroid and some geometric structures which are relevant for the purposes of this paper.
Section \ref{sec:symp} deals with symplectic sections on the prolongation associated with regular Lagrangians. In Section \ref{sec:vanish} we classify symplectic sections vanishing in the vertical subbundle. And finally, in Section \ref{sec:proof}, we complete the proof of the main theorem by constructing the relevant Poisson brackets. 
We also provide an interesting example at the end of this section.


\section{Preliminary Notions}\label{sec:prelim}

Here, we recall the notions of Lie algebroids and semisprays.
For a comprehensive treatment of Lie algebroids, see \cite{Mackenzie2005}.

\begin{definition}
A \emph{Lie algebroid} $A$ over a smooth manifold $M$ is a vector bundle $p: A \rightarrow M$ equipped with an $\mathbb{R}$-Lie bracket $[ \cdot,\cdot ]: \Gamma(A)\times \Gamma(A)\rightarrow \Gamma(A) $ on the $C^\infty(M)$-module of sections $\Gamma(A)$ and a vector bundle homomorphism $\rho : A \rightarrow TM $, called anchor map, satisfying the Leibniz rule
\begin{equation*}
[\sigma, f \eta ] =  f[\sigma,\eta] + (\mathcal{L}_{\rho(\sigma)}f)\eta,
\end{equation*}
for all $\sigma,\eta \in \Gamma(A)$, $f\in C^\infty(M).$
\end{definition}

We also denote by $\rho: \Gamma(A) \rightarrow \mathfrak{X}(M)$ the induced  $C^\infty(M)$-linear mapping on sections. Moreover, an immediate consequence of this definition is that the anchor map  $\rho$ is a Lie algebra  homomorphism:
\begin{equation*}
\rho([ \sigma, \eta ]) = [\rho(\sigma),\rho(\eta) ], \ \ \ \  \sigma,\eta\in \Gamma(A).
\end{equation*}
The image $\rho(A)= \mathcal{D}$ of the anchor map is a smooth generalized distribution on $M$ which is integrable in the sense of Stefan--Sussmann \cite{stefan-1974, sussman-1973}. 
The integral submanifolds of $\mathcal{D}$ define a singular foliation of $M$ called the \emph{characteristic foliation} of $A$.

By fixing local coordinates $(U, x^1,\ldots x^n)$ on $M$ and a local basis $\{ e_i \}_{i=1}^r$ of sections $\Gamma(U)$ we obtain local coordinates of $A$, called \emph{adapted coordinates}:
\begin{equation}
 A \ni a \mapsto (x^1,\ldots,x^n,y^1,\ldots,y^r),\label{adapcoord}
\end{equation} 
where $a = y^j \cdot e_j (x).$ The smooth functions $\rho^i_j$, $C^k_{ij}$ on $C^\infty(U)$ defined by
\begin{equation}
\rho(e_j)=\rho^i_j\frac{\partial}{\partial x^i}, \ \ \ \ [ e_i,e_j ]= C^k_{ij}e_k \label{strucfunc}
\end{equation} 
are called the \emph{structure functions} of the Lie algebroid and satisfy the so-called \emph{structure equations} of the Lie algebroid:
\begin{equation}\label{compconditions} 
\rho^i_j\frac{\partial \rho^k_l }{\partial x^i}- \rho^i_l\frac{\partial \rho^k_j }{\partial x^i}= \rho_i^kC^i_{jl}, 
\qquad
\sum_{\mathrm{cyclic}(j,l,s)}\left( \rho^i_j\frac{\partial C^k_{ls} }{\partial x^i} + C^t_{ls} C^k_{jt}\right)=0
\end{equation}

The simplest example of a Lie algebroid is the tangent bundle of a smooth manifold.

\begin{example}\label{tant-algebroid}
For any manifold $M$, the tangent bundle $\tau_M: TM \rightarrow M $ is a Lie algebroid with the standard Lie bracket of vector fields and anchor given by the identity map on $TM$.
\end{example}

Another example of Lie algebroids arises in the cotangent bundle of a \emph{Poisson manifold} \cite{coste-dazord-weinstein-1987}.

A \emph{Poisson manifold} is a smooth manifold $M$ equipped with a bivector field 
$\Pi \in \mathfrak{X}^2(M)$ satisfying the Jacobi identity
$$[\Pi,\Pi]_{\mathrm{SCH}} = 0,$$
where $[\cdot,\cdot]_{\mathrm{SCH}}$ denotes the Schouten--Nijenhuis bracket of multivector 
fields on $M$ \cite{lichnerowicz-1977}.

Equivalently, a Poisson structure on $M$ can be described by a \emph{Poisson bracket}, 
a bilinear operation $\{\cdot,\cdot\}: C^\infty(M)\times C^\infty(M)\rightarrow C^\infty(M)$ 
making $C^\infty(M)$ into a Lie algebra and satisfying the Leibniz rule
$$\{f,gh\} = \{f,g\}h + g\{f,h\}, \qquad \forall\, f,g,h\in C^\infty(M).$$
The correspondence between both descriptions is given by
$$\{f,g\} = \Pi(\mathrm{d}f,\mathrm{d}g), \qquad \forall\, f,g\in C^\infty(M).$$

\begin{example}\label{Poiss-algebroid}
Let $(M,\Pi)$ be a Poisson manifold. The cotangent bundle $T^*M$ carries a natural 
Lie algebroid structure, with anchor map given by the vector bundle morphism
$$\Pi^\sharp: T^*M \rightarrow TM, \qquad \Pi^\sharp(\omega) := \mathrm{i}_\omega \Pi$$ 
and Lie bracket on $\Omega^1(M)$ given by
$$[\sigma,\eta]_\Pi := \mathcal{L}_{\Pi^\sharp(\sigma)}\eta 
   - \mathcal{L}_{\Pi^\sharp(\eta)}\sigma - \mathrm{d}(\Pi(\sigma,\eta)),$$
for all $\sigma,\eta\in\Omega^1(M)$. This Lie algebroid over $M$ is called the 
\emph{cotangent Lie algebroid} of $(M,\Pi)$.
\end{example}


\begin{example}
An involutive distribution $\mathcal{D}\subset TM$ is a Lie algebroid with the inclusion $\mathcal{D} \hookrightarrow TM$ as anchor map and the usual bracket of vector fields as Lie bracket.
More generally, if $A$ is a given Lie algebroid and $V\subset A$ is a subbundle such that $\Gamma(V)$ is closed under the bracket on $A$, 
then $V$ is itself a Lie algebroid with anchor and bracket obtained from $A$ by simple restriction.
\end{example}

\begin{example}
Let  $\mathfrak{g}\ni \xi \mapsto \xi_M \in\mathfrak{X}(M) $ be a (right) action of a real Lie algebra $\mathfrak{g}$ on the manifold $M$. 
Its associated Lie algebroid consists of the trivial vector bundle $\mathrm{pr}_1:M \times\mathfrak{g} \rightarrow M$,
 with anchor map $\rho: M \times\mathfrak{g} \rightarrow TM $ defined by
$$\rho(m,\xi):=\xi_M(m). $$
By identifying the sections of $\mathrm{pr}_1$ with the set of $\mathfrak{g}$-valued functions $v: M \rightarrow \mathfrak{g}$, 
a Lie bracket is defined by
 $$ [v,w](m):= [ v(m),w(m)] + \langle \mathrm{d}w, (v(m))_M \rangle (m) -    \langle \mathrm{d}v, (w(m))_M \rangle (m).$$
\end{example}



\begin{definition}
A vector field $\mathcal{V}\in \mathfrak{X}(A)$ on a Lie algebroid  $ (p:A\rightarrow M, [ \cdot,\cdot], \rho) $ is called a \emph{semispray} if
\begin{equation}
Tp \circ \mathcal{V} = \rho, \label{defsemi}
\end{equation}
where $Tp: TA \rightarrow TM$ is the tangent map of the projection. 
A semispray  $\mathcal{V}\in \mathfrak{X}(A)$ is said to be a \emph{spray} 
if the vector field  $\mathcal{V}$ satisfies the homogeneity condition
\begin{equation}
m^*_k  \mathcal{V}  = k\mathcal{V}, \ \ \ \ \text{for }k>0,\label{homogcond}
\end{equation}
where $m_k: A \rightarrow A$, $m_k(a):= ka$, is the dilation map.
\end{definition}



To our knowledge, this notion was first introduced in \cite[Definition 2.2]{Weinstein1996}, and is closely related to
\emph{second-order differential equations} (SODEs) on Lie algebroids, which we recall in Section \ref{sec:Prol}.

In adapted coordinates  $(U, x, y)$ for $A$, a semispray $\mathcal{V}$  has the form
\begin{equation*}
\mathcal{V}(x,y)=y^j\rho^i_j(x)\frac{\partial}{\partial x^i}+\nu^k(x,y)\frac{\partial}{\partial y^k},
\end{equation*}
for some local functions $\nu^k(x,y)\in C^\infty (A)$. If  $ \mathcal{V}(x,y)$ is a spray, then $\nu^k(x,y)=\nu^k_{ln}(x)y^ly^n $. Note that every spray  $ \mathcal{V}$ vanishes at the zero section of the vector bundle $p: A \rightarrow M$, that is, $ \mathcal{V}(x,0)=0$.

Furthermore, the homogeneity condition \eqref{homogcond} for a spray $\mathcal{V}$ can be expressed in terms of the  Liouville (Euler) vector field $\mathbf{E}\in \mathfrak{X}(A)$, which locally has the form
\begin{equation*}
\mathbf{E}(x,y)=y^k\frac{\partial}{\partial y^k}.
\end{equation*}
Indeed, a semispray $\mathcal{V}\in \mathfrak{X}(A)$ is a spray if and only if 
\begin{equation*}
[ \mathbf{E},\mathcal{V} ]= \mathcal{V}.
\end{equation*}

Let $L \in C^\infty(A) $ be a smooth function on $A$. The \emph{Legendre transform} $\mathscr{L} :  A \rightarrow A^*$ induced by the \emph{Lagrangian} function $L$ is the map 
$$ A \ni a \mapsto   \mathscr{L}(a) \in  A^*  $$
given by 
$$\langle  \mathscr{L}(a), b \rangle := \left.  \left( \frac{\mathrm{d}}{\mathrm{d}t}L(a+tb) \right)\right|_{t=0}, $$
for all $b\in A$ such that $p(b)=p(a).$ The Lagrangian $L$ is called \emph{regular} if the Legendre transform is a diffeomorphism. In  adapted coordinates, the Lagrangian is regular if and only if the matrix 
\begin{eqnarray*}
\left(  \frac{\partial^2 L}{\partial y^i\partial y^j} \right)
\end{eqnarray*}
is non-singular.
The \emph{energy function}  $E_L \in C^\infty(A) $ associated with $L \in C^\infty(A) $ is given by

\begin{equation}
E_L := \mathcal{L}_{\mathbf{E}}L -L, \label{ener-func}
\end{equation}
where $ \mathcal{L}_{\mathbf{E}}$ denotes the Lie derivative with respect to the Liouville (Euler) vector field.

\begin{example}\label{ex:RiemannLagrangian}
  A Riemannian metric on a Lie algebroid $A$ is a section $g\in \Gamma(A^*\otimes A^*)$ that is symmetric and positive definite. 
  The corresponding Lagrangian $L \in C^\infty(A) $ is the fiberwise quadratic function $L(a):= g(a,a)$.
  In adapted coordinates, we have $L=\tfrac{1}{2}g_{ij}y^iy^j$, where $g_{ij}:=g(e_i,e_j)$.
  Thus, $\left( \frac{\partial^2 L}{\partial y^i\partial y^j} \right)=(g_{ij})$, which is non-singular since $g$ is positive definite.
  Therefore, the Lagrangian $L$ is regular.
  Furthermore, since $L$ is quadratic, we have $E_L=L$.
  The same occurs if $g$ is a pseudo-Riemannian metric.
\end{example}

\section{Main Theorem}\label{sec:main}

Now we present our main result.

\begin{theorem}\label{maintheorem}
On a Lie algebroid $(p: A \rightarrow M, [ \cdot, \cdot ], \rho)$, let  $L\in C^\infty(A)$ be a regular Lagrangian  
with associated energy function $E_L$.
Then, there exists a family of Poisson brackets $ \{ \cdot , \cdot \} $ on $A$ with the property that 
for every function $f\in C^\infty(M)$ the Hamiltonian vector field of the function $E_L +p^*f $ is a semispray.
\end{theorem}

One can deduce from Theorem \ref{maintheorem} that Hamiltonian semisprays arise naturally 
from the interaction between the Lie algebroid structure and regular Lagrangians,
giving rise to an intrinsic mechanism for generating second-order Hamiltonian dynamics in the Lie algebroid setting.

The rest of this paper is devoted to the proof of this theorem. 
Before doing so, we provide a coordinate description of the family of Poisson brackets, 
which makes the verification of their properties straightforward.

First, fix a local adapted coordinate system  $(U, x, y)=(U, x^1,\ldots,x^n,y^1\ldots,y^r)$ on $ A$ 
and consider the structure functions $\rho_j^{i},$ $C_{ij}^k$ of the Lie algebroid $A$ given by \eqref{strucfunc}.
Then, define the set of local functions ${M_{ij}}$, $M^{ij}$, and $N_{ij}$ on $A$ by
\[
M_{ij}:=\frac{\partial^2 L}{\partial y^{i}\partial y^{j}}, \ \ \ \   M_{ik}M^{kj}=\delta^{i}_j, 
\quad\text{and}\quad
N_{ij} := \rho_i^{k}\frac{\partial^2 L}{\partial x^{k}\partial y^{j}}-\rho_j^{k}\frac{\partial^2 L}{\partial x^{k}\partial y^{i}}-\frac{\partial L}{\partial y^k}C^k_{ij} + \Theta_{ij}. 
\]
where $\Theta_{ij}$ is any set of functions on $M$ satisfying
\begin{equation}\label{coefftheta}
  \Theta_{ij}(x)=-\Theta_{ji}(x), \qquad  \sum_{\mathrm{cyclic}(i,j,k)}\rho^r_k\frac{\partial \Theta_{ij}}{\partial x^r}-C^r_{jk}\Theta_{ri}=0.
\end{equation}

Now, consider the Poisson bracket on $A$ locally defined by the relations

\begin{equation}\label{coor-poissbracket}
\{  x^{i},x^{j} \}=0  ,\ \ \ \  \{  x^{i},y^{k} \}= -\rho^{i}_rM^{rk} ,\ \ \    \{  y^{k},y^{l} \}= -M^{kr}N_{rs}M^{sl}.
\end{equation}
The Hamiltonian vector field of the function
$$G(x,y)= E_L(x,y) +f(x)  = \frac{\partial L}{\partial y^k}y^k-L(x,y) + f(x) $$
is the vector field
\[
\mathcal{V}(x,y) = y^a \rho^i_a \frac{\partial}{\partial x^i} - M^{sk}\left( \rho^i_k \frac{\partial G}{\partial x^i} + N_{jk} y^j \right) \frac{\partial}{\partial y^s},
\]
which is a semispray on $A$.

In other words, given a regular Lagrangian $L$, the family of Poisson brackets in Theorem \ref{maintheorem} is given by \eqref{coor-poissbracket} 
and is such that the Hamiltonian vector field of the function $G(x,y)$ is a semispray.
This family is parameterized by the set of functions $\Theta_{ij}$ satisfying the conditions in \eqref{coefftheta}, 
which has the following geometric interpretation: 
the functions $\Theta_{ij}$ are the local coefficients of a closed section $\theta\in \Gamma(\wedge^2 A^*)$.
Indeed, as a consequence of Proposition \ref{prop2}, Corollary \ref{cor-hor-clos} and Theorem \ref{maintheorem}, we obtain that, 
for a given regular Lagrangian $L \in C^{\infty}(A)$, the family of Poisson brackets in Theorem  in \ref{maintheorem}  
can be naturally identified with the set of closed sections $\theta\in \Gamma(\wedge^2 A^*)$;
a correspondece that will be developed geometrically in Section 7.

We also notice that this procedure allows us to construct Hamiltonian sprays:
for the regular Lagrangian $L$ given by the quadratic form associated with any pseudo-Riemannian metric on the Lie algebroid,
setting $\Theta_{ij}=0$ and $f=0$ implies that $\mathcal{V}$ is a spray
(see Example \ref{ex:Spray} for more details).

The proof of the previous theorem follows from the statements:

\begin{itemize}

\item[(i)]  The set of semisprays on the Lie algebroid $(p: A \rightarrow M, [ \cdot, \cdot ], \rho)$ 
are in one-to-one correspondence with the so-called \emph{second-order differential equations}, 
which are sections of the prolongation of the Lie algebroid $A$.
Indeed, the anchor map of the prolongation maps every of second-order differential equation into a semispray on $A$.

\item[(ii)] One can characterize a family of \emph{symplectic sections} on the prolongation of the Lie algebroid  $(p: A \rightarrow M, [ \cdot, \cdot ], \rho)$ 
with the property that at least one of its Hamiltonian sections is a second-order differential equation. 

\item[(iii)]  Each member of the family of Poisson structures on $A$ in Theorem \ref{maintheorem} 
is induced by a symplectic section of (ii) through the anchor map of the prolongation.
\end{itemize}

For the rest of this paper, we develop the above statements in detail.

\section{Ehresmann Connections and Lie Algebroid Cohomology}\label{sec4}

In this section we establish the cohomological background for the construction of
the family of Poisson structures in Theorem \ref{maintheorem}.
More precisely, as we show in Proposition \ref{vanishprop}, 
for any choice of an Ehresmann connection on the prolongation of a Lie algebroid,
the resulting vertical differential is cohomologically trivial in positive degrees.
This readily implies Proposition \ref{poincare_lemma}, a sort of Poincaré Lemma for the vertical cohomology of the prolongation
that allows to globalize the constructions presented in Section \ref{sec:vanish}.

To highlight the intrinsic geometric nature of this construction and provide tools that may be of independent interest,
this section is developed in a slightly more general framework.
The proof of Proposition \ref{vanishprop} is given via an explicit construction of homotopy operators for the (isomorphic) Lie algebroid cohomology 
of a pullback vector bundle with coefficients in a Bott-type representation.
For further developments on connections on the prolongation of a Lie algebroid, see \cite{popescu2008}.

\paragraph{Lie Algebroid Cohomology and Representations.}
Let $\A\to E$ be a Lie algebroid with anchor map $\varrho:\A\to TE$ and Lie bracket $[\dcd]$.
Recall that its \emph{de Rham differential} 
$\d_\A:\Gamma(\wedge^{\bullet}\A^*)\to\Gamma(\wedge^{\bullet}\A^*)$
is the graded derivation
given by the standard Koszul formula 
\begin{eqnarray}
  \d_\A\omega(\sigma_0,\ldots,\sigma_k) &:= & \sum_{i=0}^{k}(-1)^i\mathcal{L}_{\varrho(\sigma_i)}(\omega(\sigma_0,\ldots,\hat{\sigma}_i,\ldots,\sigma_k))\nonumber \\
  && + \sum_{0\leqslant i<j \leqslant k}(-1)^{i+j}\omega([\sigma_i,\sigma_j],\sigma_0,\ldots,\hat{\sigma}_i,\ldots,\hat{\sigma}_j,\ldots,\sigma_k),\label{diff-ext-operator}
\end{eqnarray}
where $\omega\in\Gamma(\wedge^k\A^*)$ and $\sigma_0,\ldots,\sigma_k\in\Gamma(\A)$.
Moreover, we have $\d_\A^2=0$, and the cohomology of the cochain complex $(\Gamma(\wedge^{\bullet}\A^*),\d_\A)$ 
is the \emph{cohomology of the Lie algebroid} $\A$, which we denote by $H^\bullet(\A)$.
More generally, given a representation $\nabla$ of $\A$ on a vector bundle $F\to E$, that is, a flat $\A$-connection on $F$,
we get the cochain complex $(\Omega^\bullet(\A;F),\d_\nabla)$
where $\Omega^\bullet(\A;F):=\Gamma(\wedge^{\bullet}\A^*\otimes F)$ and $\d_\nabla$ is given by
\begin{eqnarray}
  \d_\nabla\lambda(\sigma_0,\ldots,\sigma_k) &:= & \sum_{i=0}^{k}(-1)^i\nabla_{\sigma_i}(\lambda(\sigma_0,\ldots,\hat{\sigma}_i,\ldots,\sigma_k))\nonumber \\
 & & + \sum_{0\leqslant i<j \leqslant k}(-1)^{i+j}\lambda([\sigma_i,\sigma_j],\sigma_0,\ldots,\hat{\sigma}_i,\ldots,\hat{\sigma}_j,\ldots,\sigma_k),\label{diffoper-connection}
\end{eqnarray}
for $\lambda\in\Omega^k(\A;F)$ and $\sigma_0,\ldots,\sigma_k\in\Gamma(\A)$ 
\cite[Chapter 7]{Mackenzie2005}.
Its cohomology is called the \emph{cohomology of the Lie algebroid} $\A$ \emph{with coefficients in the representation} $\nabla$ on $F$,
and is denoted by $H^\bullet(\A;F)$.

\subsection{Cohomology of Vertical Subbundles}

Let $E\overset{\pi}{\to}B$ be a vector bundle and $VE\subset TE$ its vertical distribution: $VE=\ker\d\pi$.
Since $VE$ is involutive, it inherits a Lie algebroid structure with the standard Lie bracket of tangent vector fields and inclusion as anchor.
In this part, we describe the Lie algebroid cohomology of $VE$ with coefficients in any representation.
We base on the fact that $VE$ is isomorphic to the pullback bundle $\pi^*E\to E$ of $E$ along $\pi: E\to B$,
which is equipped with canonical structures useful in the computation of the cohomology.

Recall that the fiber of the pullback bundle $\pi^*E\to E$ at $v\in E$ is just the fiber of $E$ at $\pi(v)$: $(\pi^*E)_v=E_{\pi(v)}$.
Moreover, $\pi^*E$ is generated by the pullback sections $\pi^*s\in\Gamma(\pi^*E)$, $s\in\Gamma(E)$,
given by $\pi^*s(v):=s(\pi(v))$.
The \emph{vertical lift} is the map
\begin{equation}\label{eq:vertlift1}
  \rho: \pi^*E \to TE, \qquad (v_0,v)\mapsto \rho(v_0,v):=v_{v_0}^{\vee}, 
\end{equation}
where $v\in E$, $v_0\in E_{\pi(v)}$, and $v_{v_0}^{\vee}\in VE_{v_0}$ is the tangent vector at $v_0$ to the curve $t\mapsto v_0+t v$:
\[  v_{v_0}^{\vee}F := \left.\tfrac{\mathrm{d}}{\mathrm{d}t}\right|_{t=0}F(v_0+t v), \qquad F\in C^\infty(E). \]

The pullback bundle $\pi^*E\to E$ is endowed with the following natural structures:
\begin{itemize}
\item A \emph{Lie algebroid structure} with the vertical lift $\rho:\pi^*E\to TE$ as anchor 
    and the Lie bracket $\brk$ on $\Gamma(\pi^*E)$ characterized by $[\pi^*s_1,\pi^*s_2]:=0$ for $s_1,s_2\in\Gamma(E)$
\item A canonical representation $\nabla$ on $\pi^*E$ satisfying $\nabla_{X}\pi^*s := 0$ for all $X\in\Gamma(\pi^*E), s\in\Gamma(E)$
\item The \emph{Euler section} $\e\in\Gamma(\pi^*E)$, given on $v\in E$ by $\e(v):=v$,
where the target $v$ is viewed in the fiber $(\pi^*E)_v=E_{\pi(v)}$
\end{itemize}
Since pullback sections generate $\pi^*E$, the Lie bracket $\brk$ and the representation $\nabla$ are well defined,
extending to arbitrary sections by $\R$-linearity and the Leibniz rule.
The vanishing of $\brk$ and $\nabla$ on pullback sections imply the flatness of $\nabla$.
Note that $\rho:\pi^*E\overset{\sim}{\to}VE$ is a Lie algebroid isomorphism mapping $\e$ to the Euler vector field on $E$,
implying that $[\e,\pi^*s]=-\pi^*s$ for all $s\in\Gamma(E)$.

\paragraph{Representations of $\bm{\pi^*E}$.}
Let $\nabla$ be a representation of $\pi^*E$ on a vector bundle $F\to E$,
with associated cochain complex $\Omega^\bullet(\pi^*E;F):=\Gamma(\wedge^\bullet\pi^*E^*\otimes F)$ of differential $\d_\nabla$.
There is an induced $\pi^*E$-connection on $\wedge^\bullet\pi^*E^*\otimes F$ 
given on $\omega\in\Omega^k(\pi^*E;F)$, and for $X,u_1, \dots, u_k \in \Gamma(\pi^*E)$ by
\begin{equation}
    (\nabla_X \omega)(u_1, \dots, u_k) = \nabla_X \big( \omega(u_1, \dots, u_k) \big) - \sum_{i=1}^k \omega(u_1, \dots, \nabla_X u_i, \dots, u_k).
\end{equation}
Here, $\nabla$ simultaneously denotes the induced connection
on $\wedge^\bullet\pi^*E^*\otimes F$, the representation of $\pi^*E$ on $F$, and
the canonical flat $\pi^*E$-connection on $\pi^*E$; the context makes clear which one is meant.

Using the insertion operator we define the Lie derivative $\Lie_\e: \Omega^\bullet(\pi^*E; F) \to \Omega^\bullet(\pi^*E; F)$ along $\e$ 
via Cartan's magic formula $\Lie_\e := i_\e \d_\nabla + \d_\nabla i_\e$.
It is straightforward to check that $\Lie_\e$ satisfies
\begin{equation}\label{eq:Lie}
  \Lie_\e \omega = \nabla_\e \omega + k \omega, \qquad \omega\in\Omega^k(\pi^*E;F).
\end{equation}

\paragraph{Scalling Cochain Map.}

For fixed $v \in E$ and $t \in (0, 1]$, consider the $\pi^*E$-path $a: [t, 1] \to \pi^*E$
over the the radial curve $\gamma(s) = sv$ on $E$ given by $a(s) = \frac{1}{s} \e(\gamma(s))$.
Then, a parallel transport operator $\mathcal{P}_t: F_{tv} \to F_v$ is defined by means of
the solution $\sigma: [t,1] \to F$ to the $\pi^*E$-covariant differential equation 
\[
\nabla_{a} \sigma = 0, \qquad \text{with initial condition } \quad \sigma(t) = \sigma_0 \in F_{tv},
\]
as $\mathcal{P}_t(\sigma_0) := \sigma(1)$.
This is a canonical isomorphism of vector spaces, that is, is independent of the choice of the path $a(s)$,
due to the flatness of $\nabla$.
Furthermore, since $a(s) = \frac{1}{s} \e(\gamma(s)) = \frac{1}{s}(sv) = v$, we have that
$a(s)$, as well as $\mathcal{P}_t$, are well defined and smooth at $t=0$.

Let $\mu_t: E \to E$ be the radial scaling map $\mu_t(v) := tv$, $t \in (0, 1]$
and consider the Lie algebroid morphism $\Phi_t: \pi^*E \to \pi^*E$ covering $\mu_t$, defined as
\begin{equation*}
    \Phi_t(u) := tu, \qquad u \in (\pi^*E)_v,
\end{equation*}
where the result $tu$ is viewed as an element of $(\pi^*E)_{tv} = E_{\pi(v)}$.
This induces a map $\Psi_t^*: \Omega^k(\pi^*E; F) \to \Omega^k(\pi^*E; F)$
given for $\omega \in \Omega^k(\pi^*E; F)$ and $u_1, \dots, u_k \in (\pi^*E)_v$ by
\begin{equation}\label{eq:Psi}
    (\Psi_t^* \omega)_v(u_1, \dots, u_k) := \mathcal{P}_t\Big(\omega_{tv}(\Phi_t u_1, \dots, \Phi_t u_k)\Big) 
    = t^k \mathcal{P}_t\Big(\omega_{tv}(u_1, \dots, u_k)\Big),
\end{equation}
where the second equality uses $\Phi_t u_i = tu_i$.
Note that, since $\Phi_t$ is a Lie algebroid morphism, 
and $\mathcal{P}_t$ is defined by $\nabla$-parallel transport along a $\pi^*E$-path,
$\Psi_t^*$ commutes with the differential $\d_\nabla$:
\begin{equation}\label{eq:commutation}
    \d_\nabla \Psi_t^* \omega = \Psi_t^* \d_\nabla \omega
\end{equation}

\begin{lemma}\label{lemma:Lie_derivative}
The map $\Psi_t^*: \Omega^k(\pi^*E; F) \to \Omega^k(\pi^*E; F)$ satisfies the evolution equation
\begin{equation}
    \frac{d}{dt} \Psi_t^* \omega = \frac{1}{t} \Psi_t^* (\Lie_\e \omega),
    \qquad
    \omega \in \Omega^k(\pi^*E; F), \quad t \in (0, 1].
\end{equation}
\end{lemma}

\begin{proof}
By differentiating the right-hand side of \eqref{eq:Psi}, we obtain from the product rule that
\begin{equation*}
\frac{d}{dt} \left((\Psi_t^*\omega)_v(u_1, \ldots, u_k)\right) = \frac{k}{t} (\Psi_t^* \omega)_v(u_1, \dots, u_k) + \frac{1}{t} (\Psi_t^* \nabla_\e \omega)_v(u_1, \dots, u_k),
\end{equation*}
where we have applied $\frac{d}{dt}\mathcal{P}_t(\sigma_0) = \mathcal{P}_t\left(\nabla_{a(t)} \sigma_0\right) = \frac{1}{t}\mathcal{P}_t\left(\nabla_\e \sigma_0\right)$, 
obtained from the definition of parallel transport.
By \eqref{eq:Lie}, we get the desired evolution equation.
\end{proof}

\paragraph{Homotopy and Cohomology.}

Given $k\geq0$, define $h_k:\Omega^k(\pi^*E; F) \to \Omega^{k-1}(\pi^*E; F)$ by 
\begin{equation}\label{eq:homot_op}
h_k \omega := \int_0^1 \frac{1}{t} \Psi_t^* (i_\e \omega) \, dt. 
\end{equation}
To see that this defines a smooth $F$-valued section, note that $h_0$ is the zero map, 
and for $k > 0$, $v \in E$ and $u_1, \dots, u_{k-1} \in (\pi^*E)_v$ that
\[ 
(\Psi_t^* i_\e \omega)_v(u_1, \dots, u_{k-1}) = t^{k-1} \mathcal{P}_t \Big( \omega_{tv}(tv, u_1, \dots, u_{k-1}) \Big) = 
t^k \mathcal{P}_t \Big( \omega_{tv}(v, u_1, \dots, u_{k-1}) \Big),
\]
so the integrand is exactly $t^{k-1} \mathcal{P}_t \Big( \omega_{tv}(v, u_1, \dots, u_{k-1}) \Big)$, 
which is smooth and bounded near $t=0$.

\begin{proposition}\label{prop:homotopy}
The maps $h_k$ in \eqref{eq:homot_op}
give a homotopy between the identity and $\Psi_0^* := \lim_{t \to 0} \Psi_t^*$:
    \[h_{k+1} \circ \d_\nabla  + \d_\nabla \circ h_k = \mathrm{id} - \Psi_0^*\]
Moreover, for $k > 0$, we have $h_{k+1} \circ \d_\nabla  + \d_\nabla \circ h_k = \mathrm{id}$, 
and, for $k=0$, we get $h_1 \circ \d_\nabla = \mathrm{id} - \mathcal{P}_0\circ p^*\circ\iota^*$,
where $\iota: B \to E$ is the zero section of $E$, and $p: E \to B$ is the bundle projection.
\end{proposition}

\begin{proof}
Given $\omega \in \Omega^k(\pi^*E; F)$, by \eqref{eq:commutation}, 
and the evolution equation of Lemma \ref{lemma:Lie_derivative}, we get
    \begin{align*}
        h_{k+1} (\d_\nabla \omega) + \d_\nabla (h_k \omega) 
        &= \int_0^1 \frac{1}{t} \Big[\Psi_t^* i_\e \d_\nabla \omega + \d_\nabla \Psi_t^* i_\e \omega \Big] \, dt \\
        &= \int_0^1 \frac{1}{t} \Psi_t^* (\Lie_\e \omega) \, dt
        = \int_0^1 \frac{d}{dt} (\Psi_t^* \omega) \, dt 
        = \Psi_1^* \omega - \lim_{t \to 0} \Psi_t^* \omega.
    \end{align*}
At $t=1$, $\Psi_1^*$ is the identity operator. Since the integral converges, $\lim_{t \to 0} \Psi_t^* \omega$ exists. Furthermore:
\begin{itemize}
    \item For $k > 0$, we get $\lim_{t\to0}(\Psi_t^* \omega)_v = \lim_{t\to0}t^k \cdot \mathcal{P}_t \circ \omega_{tv} = 0$.
    \item For $k = 0$, we have $\omega \in \Omega^0(\pi^*E; F) = \Gamma(F)$,
          so $\lim_{t\to0}(\Psi_t^* \omega)_v = \lim_{t\to0} \mathcal{P}_t(\omega_{tv}) = \mathcal{P}_0(\omega_{0_{\pi(v)}})$, 
          which is the parallel transport of $\omega$ along the radial ray from $0_{\pi(v)}$ to $v$.
\end{itemize}
\end{proof}

Consequently, the Lie algebroid cohomology vanishes in strictly positive degrees, 
and the zeroth cohomology consists of parallel sections along the fibers:

\begin{corollary}\label{cor:vanish_piE}
The Lie algebroid cohomology of $\pi^*E\to E$ with coefficients in $F\to E$ satisfies
\begin{equation*}
    H^k(\pi^*E; F) = 0 \quad \text{for } k > 0, \quad \text{and} \quad H^0(\pi^*E; F) \cong \Gamma(\iota^*F).
\end{equation*}
\end{corollary}

\begin{proof}
If $\omega \in \Omega^k(\pi^*E; F)$ is $\d_\nabla$-closed ($\d_\nabla \omega = 0$), then
it follows from Proposition \ref{prop:homotopy} that
\begin{equation*}
    \omega = \d_\nabla (h_k \omega) + \Psi_0^* \omega.
\end{equation*}
If $k > 0$, then $\Psi_0^* \omega = 0$, so $\omega$ is $\d_\nabla$-exact, and thus $H^k(\pi^*E; F) = 0$.
For $k=0$, we have $h_0=0$, and hence $\omega=\mathcal{P}_0(p^*\iota^*\omega)$, so
$\omega$ is the parallel extension of its restriction to the zero section.
Thus, the $\d_\nabla$-closed sections of $F$ are the parallel sections along the fibers of $E$.
The restriction $\Gamma(F) \to \Gamma(\iota^*F)$ induces the isomorphism $i^*: H^0(\pi^*E; F) \to \Gamma(\iota^*F)$, with inverse the parallel extension $\mathcal{P}_0\circ p^*$.
\end{proof}

\subsection{Bott Connections and Bigrading of the de Rham Differential}

Let $(\A\to E,\varrho,[\dcd])$ be a Lie algebroid, and consider an involutive subbundle $V\subset\A$ over $E$, that is, $[\Gamma(V),\Gamma(V)]\subset\Gamma(V)$.
Then, the restricted Lie bracket and anchor of $\A$ endow $V$ with a Lie algebroid structure, 
and hence, with its own de Rham differential $\d_V$.

Consider the subbundle $V^\circ\subset\A^*$ given as the annihilator of $V$ in $\A^*$, 
and the exterior algebra $\wedge^{\bullet}V^\circ\subset\wedge^{\bullet}\A^*$.
There exists a Bott-type representation of $V$ on $\wedge^{\bullet}V^\circ$, 
\[  \nabla:\Gamma(V)\times\Gamma(\wedge^{\bullet}V^\circ)\to\Gamma(\wedge^{\bullet}V^\circ), \qquad \nabla_u\alpha := \mathcal{L}_u\alpha,  \quad u\in\Gamma(V), \alpha\in\Gamma(\wedge^{\bullet}V^\circ)\] 
where the Lie derivative $\mathcal{L}_u\alpha$ is defined by the standard Cartan formula
\[  \mathcal{L}_u\alpha := \ii_u\d_A\alpha + \d_A\ii_u\alpha. \]
It is straightforward to check that $\nabla$ is indeed a representation of $V$ on $\wedge^{\bullet}V^\circ$.
In particular, for each $p\in\Z$, we have the cochain complex $(\Gamma(\wedge^\bullet V^*\otimes\wedge^{p}V^\circ),\d^{\nabla}_V)$
of the Lie algebroid $V$ with coefficients in the representation on $\wedge^{p}V^\circ$ induced by the Bott connection $\nabla$.

\paragraph{Vertical Subbundles and Ehresmann Connections.}

Note that the sections of $\wedge^\bullet V^*$ are $C^\infty(E)$-linear forms on $\Gamma(V)$,
which can be extended to sections of $\wedge^\bullet\A^*$ by the choice of a (possibly non-linear) connection.

Given the Lie algebroid $\A$ with involutive subbundle $V$, 
an \emph{Ehresmann connection} is a vector bundle projection $\gamma:\A\to\A$ on $V$, that is, $\gamma^2=\gamma$ and $\im\gamma = V$.
We say that $V$ is the \emph{vertical subbundle} associated with $\gamma$, 
and the \emph{horizontal subbundle} is $H:=\ker\gamma$.
This gives rise to a direct sum decomposition $\A=H\oplus V$,
and the decomposition of the exterior algebra
\[  \wedge^k\A^* = \bigoplus_{p+q=k} \wedge^p V^\circ \otimes \wedge^q H^\circ, \]
where $V^\circ,H^\circ\subset \A^*$ are the annihilators of $V$ and $H$, respectively.
So, each $\omega\in\Gamma(\wedge^k\A^*)$ splits as
\[  \omega = \sum_{p+q=k} \omega_{p,q}, \quad \text{with } \omega_{p,q}\in\Gamma(\wedge^p V^\circ \otimes \wedge^q H^\circ); \]
here, we say that $\omega_{p,q}$ is the component of $\omega$ of \emph{bidegree} $(p,q)$.

Note that the choice of an Ehresmann connection $\gamma$ for $V$ in $\A$
gives rise to a vector bundle isomorphism $H^\circ\cong V^*$.
Indeed, the direction $H^\circ\to V^*$ is given by simple restriction to $V$; 
while the inverse $\gamma^*:V^*\to H^\circ$ is just extending by zero on $H$.
This induces an isomorphism of vector bundles $\wedge^p V^\circ \otimes \wedge^q H^\circ\cong\wedge^q V^*\otimes\wedge^p V^\circ$ for all $p,q\in\Z$.
For each $\omega_{p,q}\in\wedge^p V^\circ \otimes \wedge^q H^\circ$, we denote by
$\bar{\omega}_{p,q}\in\wedge^q V^*\otimes\wedge^p V^\circ$ the corresponding element under this identification.

At the same time, the Ehresmann connection $\gamma$ induces a $\gamma$-dependent bigrading of the de Rham differential $\d_\A$ as
\[
  \d_\A = \dh + \dv + \dR,
\]
where the bidegrees of the components are indicated in the subscripts; that is, we have
\begin{align*}
    \d_{i,j}\big(\Gamma(\wedge^p V^\circ \otimes \wedge^q H^\circ)\big) \subseteq \Gamma(\wedge^{p+i} V^\circ \otimes \wedge^{q+j} H^\circ)
\end{align*}
for $(i,j)\in\{(1,0),(0,1),(2,-1)\}$ and all $p,q\in\Z$.
The coboundary condition $\d_\A^2=0$ translates into the following relations among the components:
\begin{align}
    \dv^2 &= 0, \label{eq:Cob02} \\
    \dh\dv + \dv\dh &= 0, \label{eq:Cob11} \\
    \dv\dR + \dR\dv + \dh^2 &= 0, \label{eq:Cob20} \\
    \dh\dR + \dR\dh &= 0, \label{eq:Cob3-1} \\
    \dR^2 &= 0 \label{eq:Cob4-2}
\end{align}
In particular, note from \eqref{eq:Cob02} that on $\Gamma(\wedge^{p}V^\circ \otimes \wedge^{\bullet}H^\circ , \dv)$
is a cochain complex for each $p\in\Z$.

\begin{lemma}\label{lemma:iso_dv}
Under the identification $\Gamma(\wedge^p V^\circ \otimes \wedge^\bullet H^\circ)\cong\Gamma(\wedge^\bullet V^*\otimes\wedge^p V^\circ)$, 
the coboundary operator $\dv$ corresponds to the de Rham differential $\d_V^\nabla$ of the Lie algebroid $V$ 
with coefficients in the Bott representation $\nabla$ on $\wedge^{p}V^\circ$.
\end{lemma}

This follows by direct inspection in the Koszul formula for $\d_\A\omega_{p,q}$
when evaluated on $p$ sections of $H$ and $q+1$ sections of $V$:
the non-vanishing terms are precisely those given by the Koszul formula in the evaluation of $\d_V^\nabla\bar{\omega}_{p,q}$.
This result is the Lie algebroid analogue of the well-known fact that 
the vertical component of the de Rham differential of a fiber bundle with respect to an Ehresmann connection 
is the de Rham differential of the fibers with coefficients in the representation given by the Bott connection
(see, for instance, \cite[Section 2]{Brahic2011}).

\paragraph{Vanishing of the $\bm{d_{0,1}}$-cohomology.}

Consider the cochain complex $(\Gamma(\wedge^p V^\circ \otimes \wedge^{\bullet}H^\circ),\dv)$ for fixed $p\in\Z$.
In the case when $V$ is isomorphic to the vertical subbundle of a vector bundle, 
we can use the previous results to show that the cohomology of this complex vanishes in positive degrees.

\begin{proposition}\label{vanishprop}
  Let $\A\to E$ be a Lie algebroid over a vector bundle $E\overset{\pi}{\to} M$,
  and $V\subset\A$ an involutive subbundle that is isomorphic as a Lie algebroid to the pullback bundle $p^*E$.
  Then, for any choice of an Ehresmann connection $\gamma$ for $V$ in $\A$, the cohomology of the cochain complex 
  $(\Gamma(\wedge^p V^\circ \otimes \wedge^{\bullet}H^\circ),\dv)$ vanishes in positive degrees, for every $p\in\Z$.
\end{proposition}

\begin{proof}
  By Lemma \ref{lemma:iso_dv}, the complex $(\Gamma(\wedge^p V^\circ \otimes \wedge^{\bullet}H^\circ),\dv)$
  is isomorphic to the cochain complex $(\Omega^\bullet(V;\wedge^p V^\circ),\d_V^\nabla)$
  of the Lie algebroid $V$ with coefficients in the Bott-type representation on $\wedge^p V^\circ$.
  In particular, their cohomologies are isomorphic.
  Since, by hypothesis, $V$ is isomorphic to $p^*E$, 
  by Corollary \ref{cor:vanish_piE}, the cohomology of $\pi^*E$ with coefficients in any representation vanishes in positive degrees, 
  and hence, so does the cohomology of $V$ with coefficients in $\wedge^p V^\circ$:
  $H^k(V;\wedge^p V^\circ) = 0$ for all $k > 0$ and $p\in\Z$.
  In particular, the $\dv$-cohomology of $\Gamma(\wedge^p V^\circ \otimes \wedge^{\bullet}H^\circ)$ vanishes in positive degrees.
\end{proof}

\section{The Prolongation of a Lie Algebroid}\label{sec:Prol}

In this section, we recall the prolongation of a Lie algebroid,
and some of its intrinsic geometric structures used in the rest of the paper.
The prolongation provides the natural framework to describe second-order dynamics 
and to construct symplectic structures associated with regular Lagrangians.

The characterization of these symplectic sections, developed in Section \ref{sec:vanish},
is the key to the proof of the main Theorem \ref{maintheorem}.
Here we develop some tools towards this characterization:
Lemma \ref{Propvertical} allows the application of the results of Section \ref{sec4},
leading in Proposition \ref{poincare_lemma} to a ``Poincaré lemma'' for the vertical cohomology of the prolongation.

Further details on the Lie algebroid prolongation can be found in \cite{carinena2007reduction,deLeonEtAl2005,martinez2002lie,Martinez2001}.

For a Lie algebroid $(p: A \rightarrow M, [ \cdot, \cdot ], \rho)$, its \emph{prolongation} over $A$  
is the vector bundle $p_A: \mathcal{T} A \rightarrow A   $ where
\begin{equation}
\mathcal{T} A :=\{ (a,b,v)\in A\times A \times TA \mid p(a)=p(b); \, \rho(b)=T_a p(v)   \}\label{prolalgebroid}
\end{equation}
and  $p_A(a,b,v):=a$.
So, the fiber over $a\in A$ is isomorphic to the vector space
\begin{equation*}
(p_A)^{-1}(a) \simeq \{ (b,v) \in A_{p(a)}\times T_a A \mid  \rho(b)= T_a p(v) \}.
\end{equation*}

Consider also the map $\hat{p}:\mathcal{T} A \rightarrow A$ given by $\hat{p}(a,b,v):=b$.
The \emph{vertical subbundle} of $\mathcal{T} A$ is
\begin{equation}\label{vert-subbundle}
\mathrm{Vert}(\mathcal{T} A ) := \mathrm{ker}\ \hat{p} = \{ (a,b,v )\in \mathcal{T} A \mid b=0, \text{ and } \, T_ap(v)=0 \}.
\end{equation}
Elements of $\mathrm{Vert}(\mathcal{T} A )$ are called \emph{vertical}. Moreover: 

\begin{definition} 
A section $Z\in \Gamma( \mathcal{T} A )$ is called \emph{vertical} if  $Z(a)\in \mathrm{Vert}(\mathcal{T} A ) $ for all $a\in A$.
\end{definition}
Notice that a vertical section $Z$ in   $\mathcal{T} A  $ takes the form
$$ Z = (\mathrm{id}_A, 0 , X),  $$
with $X\in \mathfrak{X}(A)$ a vertical vector field on $A$, i.e., $Tp \circ  X=0.$

A more general type of sections of $\mathcal{T}A$ are the so-called \emph{projectable} sections.

\begin{definition}
  A section $Z\in \Gamma( \mathcal{T} A )$ is said to be \emph{projectable} 
  if there exists $\sigma\in \Gamma(A)$ such that
\begin{equation}
\hat{p} \circ Z = \sigma \circ p.
\end{equation}
\end{definition}
One can deduce that every projectable section $Z$ has the form
$$ Z= (\mathrm{id}_A, \sigma\circ p , X), $$
where $\sigma \in \Gamma(A)$ and $X\in \mathfrak{X}(A)$ are such that $X$ and $\rho(\sigma)$ are $p$-related vector fields.

A local coordinate system on the prolongation $\mathcal{T} A $ can be obtained as follows:
fix an adapted coordinate system $(x^i, y^k)$ on $A$ associated with a local chart $(U,x^1,x^2,\ldots,x^n)$ on $M$ and a basis $\{ e_l \}$ of local sections on $A$. 
We choose a point $(a,b,v)\in \mathcal{T} A$. Since $p(a)=p(b)$ the points $a, b\in A$ have coordinates $(x^i,u^k) $ and $(x^i,y^j) $, respectively. Now, the condition $\rho(b)= T_a p(v) $ implies that $v$ has the form $\displaystyle v=y^j\rho^i_j(x)\left.\frac{\partial}{\partial x^i}\right|_a + v^l\left.\frac{\partial}{\partial y^l}\right|_a$. Thus, one can assign
\begin{equation}
\mathcal{T} A  \ni(a,b,v) \mapsto (x^i,u^k, y^j,  v^l),
\end{equation}
as coordinates. 
It follows from here that $\mathrm{rank}\,\mathcal{T} A  =2\,\mathrm{rank}\, A $ 
and  $\mathrm{rank}\,(\mathrm{Vert}(\mathcal{T} A )) = \mathrm{rank}\, A$.


One can also define a local basis $\{\mathcal{E}_i, \Upsilon_j\}$ of sections of $\mathcal{T}^A A$ 
as follows:
\begin{equation}\label{sectionprolon}
\mathcal{E}_j(a):= \left(a, e_j(p(a)),  \rho^i_j(x)\left.\frac{\partial}{\partial x^i}\right|_a  \right), \  \  \  \text{and}  \  \  \  \Upsilon_k(a):=\left(a, 0,\left.\frac{\partial}{\partial y^k}\right|_a\right)
\end{equation}
Note that these local sections are projectable.
Moreover, the sections $\Upsilon_k$ are vertical and determine a local basis for the subspace of vertical sections.

\paragraph{Lie algebroid structure of  $\mathcal{T}A$.} 
The key property that allows us to define a Lie algebroid structure on $\mathcal{T} A$ is the fact that 
every section of $\mathcal{T} A$ is locally written as a linear combination of projectable sections.
The Lie bracket can be easily defined in terms of projectable sections and extended by $\R$-linearity and the Leibniz rule.

\begin{definition}
Let $Z_i=  (\mathrm{id}_A, \sigma_i \circ p , X_i) $ be projectable sections of $\mathcal{T} A $ with $\sigma_i\in\Gamma(A) $ for $i=1,2.$  
The Lie bracket of $Z_1$ and $Z_2$ is
\begin{equation*}
[Z_1,Z_2](a):=(a,[\sigma_1,\sigma_2]\circ p(a),[X_1,X_2](a)).
\end{equation*}
\end{definition}

The Lie brackets of the elements of the basis \eqref{sectionprolon} are given by
\begin{equation}
[\mathcal{E}_i,\mathcal{E}_j ] = C^k_{ij}\mathcal{E}_k, \qquad  [\mathcal{E}_i, \Upsilon_k]=0, \quad \text{and} \quad  [\Upsilon_j, \Upsilon_k ]=0. \label{bracket-in-basis}
\end{equation}
The anchor map $\hat{\rho} : \mathcal{T}A \rightarrow TA$ is just the projection onto the third factor, $\hat{\rho}(a,b,v) :=v$,
and can also be defined locally in terms of the basis \eqref{sectionprolon} by
\begin{equation}
\hat{\rho} (\mathcal{E}_j) = \rho^i_j\frac{\partial}{\partial x^i}, \qquad   \hat{\rho} (\Upsilon_k)=\frac{\partial}{\partial y^k}. \label{anchorinbasis}
\end{equation}

Given $a,b\in A$ in the same fiber, recall that the \emph{vertical lift} $(a,b)\mapsto b_a^\vee \in T_aA$ is defined on $F\in C^\infty(A)$ by
\begin{equation*}
b_a^\vee F := \left.\frac{d}{dt}\right|_{t=0}( F(a+tb) ). 
\end{equation*}
This gives rise to the isomorphisms $\rho:p^*A\to VA\subset TA$ and $\xi^\vee :  p^*A \to \mathrm{Vert}(\mathcal{T} A)$ given by
\begin{align}
\rho(a,b) = b_a^\vee  \qquad\text{and}\qquad
\xi^\vee(a,b) =  (a,0, b_a^\vee), \label{liftmap}
\end{align}
leading to the commutative diagram
\[
\xymatrix@C=1em@R=1.5em{
p^*A \ar[rr]^{\xi^\vee} \ar[dr]_{\rho} & & \mathrm{Vert}(\mathcal{T}A) \ar[dl]^{\widehat{\rho}} \\
& VA &
}.
\]
In what follows, we refer to $\xi^\vee :  p^*A \to \mathrm{Vert}(\mathcal{T} A)$ as the \emph{vertical lifting map}.

\begin{lemma}\label{Propvertical}
Let $\mathrm{Vert}(\mathcal{T} A ) $ be the vertical subbundle \eqref{vert-subbundle}. Then:
\begin{itemize}
\item[(i)] $\mathrm{Vert}(\mathcal{T} A ) $ is an involutive subbundle,
$$ [\Gamma(\mathrm{Vert}(\mathcal{T} A ) ), \Gamma(\mathrm{Vert}(\mathcal{T} A ) )] \subset \Gamma(\mathrm{Vert}(\mathcal{T} A ) ). $$
In particular, 
$ ( p_A : \mathrm{Vert}(\mathcal{T} A ) \rightarrow A, [ \cdot , \cdot ],  \hat{\rho} : \mathrm{Vert}(\mathcal{T} A ) \rightarrow TA ) $ 
is a Lie subalgebroid of $\mathcal{T}A$.
\item[(ii)] The vertical lifting map is a Lie algebroid isomorphism.
\end{itemize}
\end{lemma}
\begin{proof}
The proof of item (i) follows directly from the bracket relations \eqref{bracket-in-basis} and the Leibniz rule. 
For the second part, the commutative diagram is the compatibility of $\xi^\vee$ with the anchor maps.
Thus, it suffices to check the bracket compatibility on a local basis of sections:
\[
\xi^\vee[p^*e_i,p^*e_j] = \xi^\vee(0) = 0 = [\Upsilon_i,\Upsilon_j] = [\xi^\vee(p^*e_i),\xi^\vee(p^*e_j)].
\]
\end{proof}


\subsection{Liouville section, vertical endomorphism and SODEs}

\begin{definition}
The \emph{Liouville (Euler) section} $\Delta\in\Gamma(\mathrm{Vert}(\mathcal{T}A))$ of the vertical subbundle
of the prolongation of the Lie algebroid $(p: A \rightarrow M, [ \cdot, \cdot ], \rho)$ is given by
\begin{equation}\label{Liuosec}
\Delta(a):=\xi^\vee(a,a)=(a,0,a^\vee_a).
\end{equation}
\end{definition}
We notice that $\hat{\rho}(\Delta)=\mathbf{E}$ is the Liouville vector field on $A$.
On the other hand, the \emph{vertical endomorphism} is the vector bundle map $J: \mathcal{T}A\rightarrow \mathcal{T}A $ given by
\begin{equation}\label{endvertical}
J(a,b,v)=(a,0,b^\vee_a).
\end{equation}

The following fact is clear due to the surjectivity of the vertical lifting map:
\begin{lemma}\label{propertiesJ}
The vertical endomorphism \eqref{endvertical} satisfies $\mathrm{Im}\, J  =\ker J = \mathrm{Vert}(\mathcal{T} A )$. Hence,  $J^2 =0$.
\end{lemma}

Let us denote by the same letter $J$ the induced endomorphism on the module of sections $J: \Gamma(\mathcal{T}A) \rightarrow \Gamma(\mathcal{T}A)$,
as well as the negative of its insertion operator $J:\Gamma(\wedge^\bullet (\mathcal{T}A)^*)\to\Gamma(\wedge^\bullet (\mathcal{T}A)^*)$,
which is a derivation of the tensor algebra of sections $\Gamma(\wedge^\bullet (\mathcal{T}A)^* ) $ characterized by
\begin{equation*}
J(L):=0, \  \  \  \ J(\sigma):=-\sigma \circ J,
\end{equation*}
for all $L\in C^\infty(A)$ and $\sigma  \in \Gamma( (\mathcal{T}A)^* )$. 
In particular, for the dual basis $\{\Upsilon^j,\mathcal{E}^i\}$ of $\{\mathcal{E}_i,\Upsilon_j\}$, we have
$$J\mathcal{E}^{i}=0 , \ \ \ \ \ \ J\Upsilon^j=-\mathcal{E}^{j},$$
and if $\omega \in \Gamma(\wedge^2 (\mathcal{T}A)^* )$  then
$$J\omega(Z_1,Z_2)=-\omega(J(Z_1),Z_2)-\omega(Z_1,J(Z_2)) $$
for all $Z_1,Z_2 \in\Gamma( \mathcal{T}A )$.

\begin{definition}
A section $Z\in \Gamma( \mathcal{T}A)$ is called a second-order differential equation (SODE) on the Lie algebroid $A$ if  
$$ J \circ Z = \Delta, $$
where $J$ is the vertical endomorphims \eqref{endvertical} and $\Delta$ is the Liouville section \eqref{Liuosec}.
\end{definition}

There is a biunivocal correspondence between SODEs and semisprays on $A$:
\begin{itemize}
\item[(i)] If $Z\in\Gamma(\mathcal{T}A) $ is a SODE then the vector field $\mathcal{V}=\hat{\rho}(Z)\in\mathfrak{X}(A)$ is a semispray.
\item[(ii)] Reciprocally, if $\mathcal{V}\in\mathfrak{X}(A)$ is a semispray then $Z(a)=(a,a,\mathcal{V}(a))$, $a\in A$, is a SODE.
\end{itemize}

From these obsevations, it is clear that $Z\in \Gamma( \mathcal{T}A) $ is a SODE on $A$ if and only if $Z$ is a section of the form
\begin{equation}
Z(a)=(a,a,\mathcal{V}(a)), 
\end{equation}
Locally, an SODE takes the following form
\begin{equation}\label{semi-coord}
Z(x,y)= y^j \mathcal{E}_j(x,y) + \nu^k(x,y) \Upsilon_k(x,y).
\end{equation}

\subsection{Non-linear connections and bigraded decomposition}

To simplify notation for the remainder of the paper, we denote by $\A:= \mathcal{T} A$
the prolongation of the Lie algebroid $A$.
Recall form Lemma \ref{Propvertical} that the vertical subbundle $\mathrm{Vert}(\A)$ is involutive in the Lie algebroid $\A$.
In this sense, following Section \ref{sec4}, an Ehresmann connection is a vector bundle morphism $\gamma:\A\to\A$ 
such that $\operatorname{Im}\gamma=\mathrm{Vert}(\A)$ and $\gamma^2=\gamma$,
giving rise to a horizontal subbundle $\mathbb{H}$, complementary to the vertical bundle $\mathrm{Vert}(\A)$, on $\A$: 
$\A=\mathrm{Vert}(\A)\oplus\mathbb{H}$.

For a fixed basis  $\{ \mathcal{E}_i, \Upsilon_j \}$ of local sections on $\A$ as in \eqref{sectionprolon},  
with dual basis $\{\Upsilon^j,\mathcal{E}^i\}$, 
an Ehresmann connection is locally given by
\begin{equation}
\gamma=\operatorname{Id}_\A-\mathcal{E}^i\otimes(\mathcal{E}_i-\gamma^j_i\Upsilon_j).
\end{equation}
The vertical and horizontal subbundles and their annihilators are locally generated as
$$ \mathrm{Vert}(\mathscr{A})= \mathrm{span}\{\Upsilon_j \}, \ \ \ \   \mathbb{H}=  \mathrm{span}\{\mathrm{hor}_{i}:=\mathcal{E}_i-\gamma_i^{j}\Upsilon_j \} , $$
$$ \mathrm{Vert}(\mathscr{A})^{0}= \mathrm{span}\{\mathcal{E}^{j} \}, \ \ \ \   \mathbb{H}^{0}=  \mathrm{span}\{\nu^{i}:=\Upsilon^{i}+\gamma_j^{i}\mathcal{E}^{j} \}.$$
The choice of an Ehresmann connection $\gamma$ induces a bigrading of $\Gamma(\wedge^k \mathscr{A}^*)$,
\begin{equation}
\Gamma(\wedge^k \mathscr{A}^*) = \bigoplus_{p+q=k} \Gamma_{pq}(\wedge^k \mathscr{A}^* ),\label{bigrad}
\end{equation}
where $\Gamma_{pq}(\wedge^k \mathscr{A}^*):=\Gamma(\wedge^p\mathrm{Vert}(\mathscr{A})^{0}\otimes\wedge^q\mathbb{H}^0)$ 
denotes the set of sections of bidegree $(p,q)$.

The de Rham differential $\mathrm{d}_\mathscr{A}$ also has the bigraded decomposition
\begin{equation}
\mathrm{d}_\mathscr{A}= \mathrm{d}_\mathscr{A}'+\mathrm{d}_\mathscr{A}''+\partial_\mathscr{A},\label{decom-extd}
\end{equation}
where each component $\mathrm{d}_\mathscr{A}'$, $\mathrm{d}_\mathscr{A}''$, and $\partial_\mathscr{A}$
is of bidegree $(1,0)$, $(0,1)$, and $(2,-1)$, respectively. 
The coboundary condition $\mathrm{d}_\mathscr{A} \circ \mathrm{d}_\mathscr{A} =0 $ is equivalent to the relations
\begin{eqnarray*}
\mathrm{d}_\mathscr{A}''\circ \mathrm{d}_\mathscr{A}'' = 0, \ \ \ \partial_\mathscr{A}\circ \partial_\mathscr{A} =0, \ \ \ \ \mathrm{d}_\mathscr{A}'\circ \mathrm{d}_\mathscr{A}''+\mathrm{d}_\mathscr{A}''\circ\mathrm{d}_\mathscr{A}'=0,\\
\mathrm{d}_\mathscr{A}'\circ \partial_\mathscr{A}+  \partial_\mathscr{A}\circ \mathrm{d}_\mathscr{A}'=0, \ \ \ \ \mathrm{d}_\mathscr{A}'\circ \mathrm{d}_\mathscr{A}'+ \mathrm{d}_\mathscr{A}''\circ\partial_\mathscr{A}  + \partial_\mathscr{A} \circ \mathrm{d}_\mathscr{A}''=0.
\end{eqnarray*}

Observe that the component $\mathrm{d}_\mathscr{A}''$ in the decomposition \eqref{decom-extd} is also a coboundary operator.
Furthermore, as discussed in Section \ref{sec4}, its cohomology is trivial in positive degrees.

\begin{proposition}\label{poincare_lemma}
Let $\vartheta\in \Gamma_{pq}(\wedge^k \mathscr{A}^* )$ be a section of bidegree $(p,q)$, $q\geq1$. 
If $\mathrm{d}_\mathscr{A}''\vartheta=0$, then there exists $\zeta\in \Gamma_{p,q-1}(\wedge^{k-1} \mathscr{A}^*)$ such that
$$ \mathrm{d}_\mathscr{A}'' \zeta = \vartheta. $$
\end{proposition}
\begin{proof}
For the Lie algebroid $\A\to A$, the vertical subbundle $\mathrm{Vert}(\A)$ involutive
and isomorphic to the pullback bundle $p^*A$, due to Lemma \ref{Propvertical}.
So, by Proposition \ref{vanishprop}, the cohomology of 
$$(\Gamma(\wedge^p\mathrm{Vert}(\mathscr{A})^{0}\otimes\wedge^\bullet\mathbb{H}^0),\mathrm{d}_\mathscr{A}'')$$ 
vanishes in degree $q$, meaning that every $\vartheta\in \Gamma_{pq}(\wedge^k \mathscr{A}^*)$ with $\mathrm{d}_\mathscr{A}''\vartheta=0$
is of the form $\vartheta = \mathrm{d}_\mathscr{A}'' \zeta$ for some $\zeta\in \Gamma_{p,q-1}(\wedge^{k-1} \mathscr{A}^*)$.
\end{proof}


\section{Symplectic Sections on the Prolongation}\label{sec:symp}

In this section, we recall the notion of \emph{symplectic Lie algebroid} \cite{deLeonEtAl2005}, 
that is, a Lie algebroid equipped with a covariant 2-section with properties analogous to those of symplectic structures.
In particular, we recall their construction on the prolongation as the \emph{Cartan 2-section} of a regular Lagrangian.

\subsection{Symplectic Lie Algebroids}

Given a Lie algebroid $(p: A\to M, [\cdot ,\cdot]_A, \rho)$, with de Rham differential $\mathrm{d}_A$,
a section $\omega \in \Gamma(\wedge^k A^* )$ is said to be \emph{closed} if $\mathrm{d}_A\omega =0 $.
If $\omega \in \Gamma(\wedge^2 A^*)$, then we have the vector bundle morphism $\omega^\flat: A \rightarrow A^*$ 
given by the insertion: $a \mapsto \mathrm{i}_a\omega$. 
We say that $\omega$ is \emph{non-degenerate} if $\omega^\flat$ is an isomorphism of vector bundles.

\begin{definition}
A \emph{symplectic structure} on the Lie algebroid $A \rightarrow M$ is a 
closed and non-degenerate section $\omega \in \Gamma(\wedge^2 A^* )$.  
In this case, the pair $(A, \omega)$ is called a \emph{symplectic Lie algebroid}.
\end{definition}


On a symplectic algebroid $(A, \omega)$, for each function $f\in C^\infty(M)$ there exists a unique $\sigma_f \in \Gamma( A ) $
called the \emph{Hamiltonian section} of $f$, such that
\begin{equation}
\mathrm{i}_{\sigma_f} \omega = -\mathrm{d}_A f.
\end{equation}
Furthermore, as is well known, 
every symplectic Lie algebroid gives rise to a Poisson structure on the base manifold $M$.
This can be seen as an instance of Proposition 3.6 and Theorem 4.1 in \cite{MackenzieXu1994} 
in the context of triangular Lie bialgebroids.
For completeness, we present the following proof.

\begin{proposition}\label{poisson-base}
Let $(A, \omega)$ be a symplectic algebroid. The symplectic structure $\omega$ defines a Poisson bracket on $M$ given by
\begin{equation}
\{f, g  \}_\omega := \omega(\sigma_f,\sigma_g), \qquad f,g\in C^\infty(M). \label{pois-brac}
\end{equation}
Furthermore, the anchor map relates the Hamiltonian section $\sigma_f$ and vector field $X_f$ as
\begin{equation*}
 \rho(\sigma_f)= X_f
\end{equation*}
\end{proposition}

\begin{proof}
  It is clear that \eqref{pois-brac} defines a bivector field $\pi$ on $M$.
  Furthermore, if $\pi_\omega\in\Gamma(\wedge^2A)$ is the Poisson structure on $A$ induced by $\omega$,
  then \eqref{pois-brac} implies that $\pi_\omega$ and $\pi$ are $\rho$-related: $\pi^\sharp = \rho\circ\pi_\omega^\sharp\circ\rho^{*}$.
  Since the anchor map $\rho:\Gamma(A)\to \mathfrak{X}(M)$ is a Lie algebra homomorphism, 
  its extension to multisections $\wedge^{\bullet}\rho:\Gamma(\wedge^{\bullet}A)^*\to\mathfrak{X}^\bullet(M)$ 
  is a graded Lie algebra homomorphism when equipped with their respective Schouten--Nijenhuis brackets.
  The $\rho$-relatedness of $\pi$ and $\pi_\omega$ then implies that 
  \[
    [\pi,\pi] = [\wedge^2\rho\pi_\omega,\wedge^2\rho\pi_\omega]= \wedge^3\rho[\pi_\omega,\pi_\omega]=0.
  \]
  For the second part, note for every $g\in C^\infty(M)$ that
  \[
    \mathcal{L}_{\rho(\sigma_{f})}g = \langle \mathrm{d}_A g ,  \sigma_{f} \rangle = \omega ( \sigma_{f},\sigma_{g} ) = \{f, g  \}_\omega = \mathcal{L}_{X_f}g. 
  \]
\end{proof}

\begin{example}
For the tangent Lie algebroid $(\tau_M: TM \rightarrow M) $, a symplectic section $\omega \in \Gamma(\wedge^2 T^*M) $ 
is precisely a symplectic 2-form on $M$ and the symplectic algebroid $(TM, \omega)$ turns out to be a symplectic manifold $(M, \omega)$. 
\end{example}

\subsection{The Cartan 2-Section of a Regular Lagrangian}

In this part, we describe how one can construct a symplectic section on the prolongation $\A:=\mathcal{T}A$ of a Lie algebroid $(p: A \rightarrow M, [ \cdot, \cdot ], \rho)$ 
by means of a regular Lagrangian $L\in C^\infty(A)$. 
To this end, we first introduce the Poincaré--Cartan 1-section.

\begin{definition}
Given a function $ L \in C^\infty(A)$, the \emph{Cartan 1-section} $\theta_L\in\Gamma(\A^*)$ is defined by
\begin{equation*}
\theta_L := -J(\mathrm{d}_{\A}L).
\end{equation*}
\end{definition}

The Cartan 1-section is horizontal, that is, it vanishes on $\mathrm{Vert}( \A )$. 
Locally, it takes the form
\begin{equation*}
\theta_L = \frac{\partial L}{ \partial y^k} \mathcal{E}^k.
\end{equation*}

\begin{definition}
The Cartan 2-section is the differential of $\theta_L$:
\begin{equation}\label{cartantwo}
\omega_L := \mathrm{d}_{\A}\theta_L
\end{equation}
\end{definition}

The local expression of the Cartan 2-section is
\begin{equation}\label{Cartan-2fomcoord}
  \omega_L = M_{ij}\Upsilon^i\wedge \mathcal{E}^j + \frac{1}{2}N_{kl}\mathcal{E}^k\wedge \mathcal{E}^l,
\end{equation}
where
\[
  M_{ij} := \frac{\partial^2 L}{\partial y^i\partial y^j}
  \quad\text{and}\quad
  N_{kl} :=  \rho^r_k\frac{\partial^2 L}{\partial x^r\partial y^l}-\rho^r_l\frac{\partial^2 L}{\partial x^r\partial y^k}-\frac{\partial L}{\partial y^r}C^r_{kl}.
\]
From here, it is clear that $\omega_L$ vanishes on the vertical subbundle, i.e.,   
$$\left. \omega_L \right|_{\mathrm{Vert}(\A) \times \mathrm{Vert}(\A)}=0.$$
Moreover, it follows form \eqref{Cartan-2fomcoord} that $ L \in C^\infty(A)$ is a regular Lagrangian if and only if
the Cartan 2-section $\omega_L $ \eqref{cartantwo} is a symplectic section.
In this case, since $\operatorname{rank}\A=2\operatorname{rank} \mathrm{Vert}(\A)$, 
we have that $\omega_L^{\flat}:\mathrm{Vert}(\A)\to \mathrm{Vert}(\A)^0$ is an isomorphism.
Furthermore, the Hamiltonian section of its energy function
\begin{equation}\label{energyfunc}
E_L := \mathcal{L}_{\mathbf{E}} L - L = \frac{\partial L}{\partial y^k}y^k-L
\end{equation}
is a second-order differential equation.

\begin{proposition}\label{prop:sigmaSODE}
Let $ L \in C^\infty(A)$ be a regular Lagrangian with Cartan 2-section $\omega_L$. 
Then, the Hamiltonian section $\sigma_{E_L}$ of the energy function $E_L$,
\begin{equation}\label{simpleceqn}
\mathrm{i}_ {\sigma_{E_L}}\omega_L =-\mathrm{d}_{\A}E_L,
\end{equation}
is a SODE on $A$.
\end{proposition}

\begin{proof}
We shall prove that $\sigma_{E_L}$ has locally the form \eqref{semi-coord}. 
Let $\sigma_{E_L} = f^a\mathcal{E}_a+g^b\Upsilon_b$. By \eqref{Cartan-2fomcoord}, 
\begin{equation*}
\mathrm{i}_{\sigma_{E_L}} \omega_L = -  M_{ia}f^a \Upsilon^i+(M_{bj}g^b+f^aN_{aj})\mathcal{E}^j.
\end{equation*}
Since 
\begin{equation}\label{eq:dAEL}
\mathrm{d}_{\A}E_L =  M_{ia} y^a \Upsilon^i + \rho^k_j\frac{\partial E_L}{\partial x^k}\mathcal{E}^j,
\end{equation}
we have from \eqref{simpleceqn} that
\begin{equation}
 M_{ia}( y^a-f^a)=0, \quad\text{for all $i$.}  \label{eqsymp1}
\end{equation}
The regularity of $L$ implies that the unique solution of \eqref{eqsymp1} is $f^a=y^a$, so $\sigma_{E_L}$ is a SODE.
\end{proof}

\begin{corollary}\label{cor1}
If $L \in C^\infty(A)$ is a regular Lagrangian,
then the Poisson structure $\{\cdot,\cdot\}_{ \omega_L }$ on $A$ associated with the Cartan 2-section $\omega_L$
has as Hamiltonian vector field the semispray $\mathcal{V}_L:=\hat{\rho}(\sigma_{E_L})$.
\end{corollary}

This corollary is a particular instance of Theorem \ref{maintheorem}. 
Actually, in our proof of Theorem  \ref{maintheorem}, we construct a family of Poisson structures on $A$
that contains that of Corollary \ref{cor1}.

Further computations from the proof of Proposition \ref{prop:sigmaSODE} 
show that the local expressions of the Hamiltonian section $\sigma_{E_L}$ and the semispray $\mathcal{V}_L$ are
\begin{align*}
\sigma_{E_L} = y^\alpha \mathcal{E}_\alpha - M^{rl}\left(\rho^b_l\frac{\partial E_L}{\partial x^b} + N_{sl}y^s \right)\Upsilon_r,
\qquad
\mathcal{V}_L = y^\alpha\rho^\beta_\alpha\frac{\partial}{\partial x^\beta} - M^{rl}\left(\rho^b_l\frac{\partial E_L}{\partial x^b} + N_{sl}y^s \right)\frac{\partial}{\partial y^r}.
\end{align*}

These formulae allow us to give positive answers to the question of existence of Hamiltonian sprays on $A$:

\begin{example}\label{ex:Spray}
  Let $g\in\Gamma(A^*\otimes A^*)$ be a pseudo-Riemannian metric on the Lie algebroid $(A,\rho,[\dcd])$
  and consider the regular (quadratic) Lagrangian $L\in C^\infty(A)$ of Example \ref{ex:RiemannLagrangian}: $L(a) := g(a,a)=\tfrac{1}{2}g_{rs}(x)y^r y^s$.
  Notice that the Hamiltonian semispray $\mathcal{V}_L$ in Corollary \ref{cor1} is actually a spray, since
  \[
    M^{sk} = g^{sk} 
    \quad\text{and}\quad
    N_{jk} = \left(\rho^r_j\frac{\partial g_{sk}}{\partial x^r}-\rho^r_k\frac{\partial g_{sj}}{\partial x^r}-g_{is}C^i_{jk}\right)y^s,
  \]
  where $g^{rs}$ is the inverse of the metric $g_{rs}$ on $A$.
  From here, it is clear that the coefficients of $\frac{\partial}{\partial y^s}$ in $\mathcal{V}_L$ are fiberwise quadratic,
  implying that $\mathcal{V}_L$ is an spray on $A$.
  We also observe that $N_{jk}$ are the coefficients of $d_A(g^{\flat}\mathbf{E})$,
  where $\mathbf{E}$ is the Euler vector field.
  In terms of the Christoffel symbols $\Gamma^{\mu}_{\nu\lambda}$ of the Levi-Civita $A$-connection of $g$, 
  the Hamiltonian spray $\mathcal{V}_L$ is
  \[
    \mathcal{V}_L = y^\alpha\rho^\beta_\alpha\frac{\partial}{\partial x^\beta} + 
    g^{sk}\left(\left( -\Gamma_{ab}^{\mu} g_{\mu k} + \tfrac{1}{2} \Gamma_{kb}^{\mu} g_{\mu a} - \tfrac{1}{2} \Gamma_{ka}^{\mu} g_{\mu b} \right) y^a y^b\right)\frac{\partial}{\partial y^s}.
  \]
\end{example}

\section{Symplectic Sections Vanishing on the Vertical Subbundle}\label{sec:vanish}

In this section, we study the class of symplectic sections on the prolongation of a Lie algebroid that vanish on the vertical subbundle. 
Concretely, we present two key results. 
In Proposition \ref{prop1}, we characterize those symplectic sections vanishing whenever both arguments are vertical. 
Then, in Proposition \ref{prop2}, we describe the subfamily annihilated by the dual extension of the vertical endomorphism.
This class of symplectic sections induce Poisson brackets on the Lie algebroid $A$
having the properties stated in Theorem \ref{maintheorem}, namely,
admitting a Hamiltonian semispray on $A$.


Recall that a section $Z\in \Gamma(\mathscr{A})$ is called vertical if $Z(a)\in \mathrm{Vert}(\mathscr{A})$ for all $a\in A$. 
A section $\omega \in \Gamma(\wedge^k \mathscr{A}^*)$ is called \emph{horizontal} if $\mathrm{i}_Z\omega =0$ for every vertical section $Z$. 

\begin{definition}
Let $\Omega \in \Gamma(\wedge^2 \mathscr{A}^*) $ be a symplectic section. We say that $\Omega$ vanishes on the vertical subbundle if  $\left. \Omega \right|_{\mathrm{Vert}(\mathscr{A}) \times \mathrm{Vert}(\mathscr{A})}=0.$
\end{definition}

Now, we characterize the symplectic sections vanishing on the vertical subbundle.

\begin{proposition}\label{prop1}
If a symplectic section $\Omega \in \Gamma(\wedge^2 \mathscr{A}^* ) $ vanishes on the vertical subbundle, then
\begin{equation}
\Omega = \Theta + \mathrm{ d}_{\mathscr{A}} \zeta, 
\end{equation}
where $\zeta\in\Gamma(\mathscr{A}^*)$ and $\Theta\in\Gamma(\wedge^2 \mathscr{A}^*)$ are horizontal, $\Theta$ is $\mathrm{d}_{\mathscr{A}}$-closed and $\mathrm{rank}(\mathrm{d}_{\mathscr{A}} \zeta)=2r$.
\end{proposition}
\begin{proof}
Let $\Omega \in \Gamma(\wedge^2 \mathscr{A}^* ) $ be a symplectic section vanishing on the vertical subbundle. 
Fix any Ehresmann connection $\gamma$ and assume that $\Omega$ has the following bigrading decomposition
$$\Omega= \Omega_{20}+\Omega_{11} +\Omega_{02}.$$
Since $\Omega$ vanishes on the vertical subbundle, $\Omega_{02}\equiv 0$. Using the bigraded decomposition \eqref{decom-extd} of the exterior diferential $\mathrm{ d}_{\mathscr{A}} $ and the closedness of $\Omega$, we get
\begin{eqnarray}
\mathrm{ d}_{\mathscr{A}} '' \Omega_{11}=0, \label{clos1}\\
\mathrm{ d}_{\mathscr{A}} ' \Omega_{11}+ \mathrm{ d}_{\mathscr{A}} '' \Omega_{20}=0,\label{clos2}\\
\mathrm{ d}_{\mathscr{A}} ' \Omega_{20}+\partial_\mathscr{A} \Omega_{11} =0.\label{clos3}
\end{eqnarray}
By \eqref{clos1} and Proposition \ref{poincare_lemma}, there exists $\zeta \in\Gamma_{10}( \mathscr{A}^*)  $ such that
$ \mathrm{ d}_{\mathscr{A}} ''   \zeta =  \Omega_{11}$. 
Now, from equations \eqref{clos2} and \eqref{clos3},  we have 
$$\mathrm{ d}_{\mathscr{A}} ' \Omega_{11}+ \mathrm{ d}_{\mathscr{A}} '' \Omega_{20}= \mathrm{ d}_{\mathscr{A}} '  \mathrm{ d}_{\mathscr{A}} ''   \zeta+ \mathrm{ d}_{\mathscr{A}} '' \Omega_{20}=\mathrm{ d}_{\mathscr{A}} ''(\Omega_{20}- \mathrm{ d}_{\mathscr{A}} '  \zeta) =0,  $$
$$\mathrm{ d}_{\mathscr{A}} ' \Omega_{20}+\partial_\mathscr{A} \Omega_{11}= \mathrm{ d}_{\mathscr{A}} ' \Omega_{20}-\mathrm{ d}_{\mathscr{A}} '\circ\mathrm{ d}_{\mathscr{A}} '   \zeta -\mathrm{ d}_{\mathscr{A}} '' \partial_\mathscr{A}  \zeta   = \mathrm{ d}_{\mathscr{A}} '(\Omega_{20}- \mathrm{ d}_{\mathscr{A}} ' \zeta )  =0. $$
These relations imply that $\Theta:= \Omega_{20}- \mathrm{ d}_{\mathscr{A}} ' \zeta$ is a horizontal closed section.  
Finally,
$$\Omega= \Omega_{20}+\Omega_{11} +\Omega_{02}= \Theta+ \mathrm{ d}_{\mathscr{A}} ' \zeta+ \mathrm{ d}_{\mathscr{A}}''\zeta = \Theta + \mathrm{ d}_{\mathscr{A}} \zeta .$$
\end{proof}

\begin{proposition}\label{prop2}
Let $\Omega \in \Gamma(\wedge^2 \mathscr{A}^* ) $ be a symplectic section vanishing on the vertical subbundle. Then, $J(\Omega)=0$ if and only if 
$$ \Omega = \Theta + \omega_L, $$ 
where $\omega_L$ is the Cartan 2-section for some regular Lagrangian $L\in C^\infty(A)$.


\end{proposition}
\begin{proof}
First, we shall assume that $J(\Omega) =0$. By Proposition \ref{prop1}, $\Omega = \Theta + \mathrm{ d}_{\mathscr{A}} \zeta$. 
Since $\Theta$ and $\zeta$ are horizontal, Proposition \ref{propertiesJ} implies that
$$ J(\Omega) = J( \mathrm{ d}_{\mathscr{A}}'' \zeta  ) .$$
Then $J( \mathrm{ d}_{\mathscr{A}}'' \zeta  )=0$. 
Since $J:\mathbb{H}^0\to\mathrm{Vert}(\mathscr{A})^0$ is an isomorphism, 
there exists a unique $(0,1)$-section $\widetilde{\zeta}$ such that $J(\widetilde{\zeta})=-\zeta$.
Furthermore, we have that $\mathrm{ d}_{\mathscr{A}}'' \widetilde{\zeta} =0 $. 
Indeed, If $\zeta =\zeta^{j}\mathcal{E}^{j}$, then $\widetilde{\zeta} = -\zeta^{j}\nu^{j}$ and, by a direct computation, we get 
$$ \mathrm{ d}_{\mathscr{A}}'' \widetilde{\zeta} = \frac{\partial \zeta^{j} }{\partial {y}^{i} }\nu^{i}\wedge\nu^{j} \ \ \ \ \text{and}\ \ \ \ J( \mathrm{ d}_{\mathscr{A}}'' \zeta  )= -\frac{\partial \zeta^{j} }{\partial {y}^{i} }\mathcal{E}^{i}\wedge \mathcal{E}^{j}.$$
From here, we get that the vanishing of $J(\mathrm{d}''_\A\zeta)$ implies $\mathrm{ d}_{\mathscr{A}}'' \widetilde{\zeta} =0$.
Thus, by Proposition \ref{poincare_lemma}, there exists $L\in C^\infty(A)$ such that $\mathrm{ d}_{\mathscr{A}}'' L =\widetilde{\zeta}$. 
Since $J$ vanishes on $\mathrm{Vert}(\mathscr{A})^0$, we get
\[
\d_\A\zeta = -\d_\A J(\d''_\A L) = -\d_\A J(\d_\A L) = \omega_L.
\]
Thus, $ \Omega = \Theta + \omega_L$. 
Since $\Omega$ is symplectic, the Lagrangian $L$ has to be regular. 
Conversely, if $ \Omega = \Theta + \omega_L $, then by \eqref{Cartan-2fomcoord} we get 
$$ J(\Omega) = J(\omega_L) = -\frac{\partial^2 L}{\partial y^i\partial y^j}\mathcal{E}^i\wedge\mathcal{E}^j, $$
which vanishes since $\frac{\partial^2 L}{\partial y^i\partial y^j}$ is symmetric while $\mathcal{E}^i\wedge\mathcal{E}^j$ is skew-symmetric on $i$ and $j$.
\end{proof}

\paragraph{Characterization of Horizontal Closed Sections in the Prolongation.}

Following \cite{Martinez2001}, recall that the map $\hat{p}:\A:=\mathcal{T} A \rightarrow A$ is a Lie algebroid epimorphism over the projection $p:A\rightarrow M$.
In consequence, we have that $\hat{p}^{\,*}$ is a cochain complex monomorphism: $\hat{p}^{\,*} \circ \d_{\A} = \d_A \circ \hat{p}^{\,*}$.
Furthermore, since $\mathrm{Vert}(\A)=\ker\hat{p}$, 
it follows that $\hat{p}^{\,*}:\wedge^{\bullet}A^*\to\wedge^{\bullet}\mathrm{Vert}(\A)^\circ$ is an isomorphism.

\begin{lemma}\label{lemma-hor-closed}
  Let $\Theta\in\Gamma(\wedge^k\mathrm{Vert}(\A)^{\circ})$ be a horizontal section. 
  Then, $\Theta$ is $\d_\A$-closed if and only if
  there exists a $\d_A$-closed section $\theta\in\Gamma(\wedge^kA^*)$ such that $\hat{p}^{\,*}\theta = \Theta$.
\end{lemma}

\begin{proof}
  Since $\hat{p}^{\,*}:\wedge^{\bullet}A^*\to\wedge^{\bullet}\mathrm{Vert}(\A)^\circ$ is an isomorphism,
  there exists a unique $\theta\in\Gamma(\wedge^kA^*)$ such that $\hat{p}^{\,*}\theta = \Theta$.
  Taking into account that $\hat{p}^{\,*}$ is a morphism of cochain complex, we have that $\hat{p}^{\,*}\d_A\theta = \d_{\A}\Theta$.
  From here, it is clear that $\Theta$ is $\d_{\A}$-closed if $\d_A\theta = 0$.
  Conversely, if $\d_{\A}\Theta = 0$, then $\hat{p}^{\,*}\d_A\theta = 0$.
  Since $\hat{p}^{\,*}$ is injective, we conclude that $\d_A\theta = 0$.
\end{proof}

The preceding proposition provides a simple intrinsic description of the horizontal closed sections of the prolongation 
and allows us to construct them systematically.
With this in mind, we arrive at the following characterization of the symplectic sections vanishing on the vertical subbundle by means of objects on the base algebroid.

\begin{corollary}\label{cor-hor-clos}
A symplectic section $\Omega\in\Gamma(\wedge^2\A^*)$ vanishing on the vertical subbundle satisfies $J(\Omega)=0$ if and only if
there exist a regular Lagrangian function $L\in C^\infty(A)$ and a $\d_A$-closed section $\theta \in \Gamma(\wedge^2 A^*)$ such that
\[
  \Omega = \omega_L + \hat{p}^{\, *}\theta.
\]
\end{corollary} 

This characterization provides the geometric framework underlying the parametrization of the family of Poisson structures introduced in Section \ref{sec:main}. 
In particular, the closed sections obtained here correspond to the data that determine the family of Poisson brackets in Theorem \ref{maintheorem}. 
The precise relation between these objects will be established in Section \ref{sec:proof}.

\section{Proof of the Main Theorem}\label{sec:proof}

In this section, we present a proof of Theorem \ref{maintheorem} by using 
the characterizations of symplectic sections on the prolongation vanishing in the vertical distribution obtained in Section \ref{sec:vanish}.
We show that the Poisson brackets induced by these sections on the Lie algebroid have the properties stated in the theorem.
More precisely, we prove that each symplectic section satisfying the conditions of Proposition \ref{prop2}
gives rise to a Poisson bracket having semisprays Hamiltonian vector fields.
This establishes the link between the geometric framework developed in Section \ref{sec:vanish} 
and the coordinate description introduced in Section \ref{sec:main}.

We begin the proof by choosing  a regular Lagrangian $L\in C^\infty(A)$ 
and consider the Cartan 2-form $\omega_L$ \eqref{cartantwo} associated with $L$. Let
\[
  \Omega = \Theta + \omega_L,
\]
where $\Theta$ is an arbitrary $\d_\A$-closed horizontal section. 
By Proposition \ref{prop2}, $\Omega$ is a $\Theta$-parameterized family of symplectic sections 
vanishing on the vertical subbundle such that $J(\Omega)=0$. 

Let $\{ \cdot,\cdot \}_{\Omega}$ be the Poisson bracket \eqref{pois-brac} on $A$ induced by $\Omega$.
To show that the Hamiltonian vector field of $E_L +p^*f$, $f\in C^\infty(M)$, is a semispray,
we shall show that it corresponds to a Hamiltonian section of $\Omega$ which is also an SODE.  

Consider the Hamiltonian section $\sigma_{E_L} $ with respect to $\omega_L$,
$$ \mathrm{i}_{\sigma_{E_L} }\omega_L = - \mathrm{ d}_{\mathscr{A}} E_L.$$
Recall from Proposition \ref{prop:sigmaSODE} that $\sigma_L$ is a SODE on $A$.
For $f\in C^\infty(M)$, let $Z\in \Gamma(\mathscr{A})$ be the section given by
$$\mathrm{i}_Z \omega_L = -\mathrm{ d}_{\mathscr{A}}(p^*f) - \mathrm{i}_{\sigma_L }\Theta,$$
which is well-defined because of $\omega_L$ is non-degenerate. 
Since $\omega_L^{\flat}:\mathrm{Vert}(\A)\to \mathrm{Vert}(\A)^0$ is an isomorphism,
we have that $Z$ is a vertical section.
Taking into account that $\sigma_{E_L}$ is a SODE, the sum $Z+ \sigma_{E_L}$ is again an SODE. 
By direct computations, we get
\[
\mathrm{i}_{Z +\sigma_{E_L}}\Omega = \mathrm{i}_{\sigma_{E_L} }\Theta + \mathrm{i}_{Z +\sigma_{E_L}}\omega_L = \mathrm{i}_{\sigma_{E_L} }\Theta - \mathrm{ d}_{\mathscr{A}}(p^*f) - \mathrm{i}_{\sigma_{E_L} }\Theta - \mathrm{ d}_{\mathscr{A}} E_L
= - \mathrm{ d}_{\mathscr{A}}( E_L + p^*f).
\]
Thus, $Z +\sigma_{E_L} $ is both an SODE on $A$ and a Hamiltonian section for $\Omega$. 
By Proposition \ref{poisson-base}, the vector field $\mathcal{V}:=\hat{\rho}(\sigma_{E_L} + Z)$ is both a semispray on $A$ and
the Hamiltonian vector field of the function the function $E_L+p^*f$ with respect to the Poisson structure $\{ \cdot,\cdot \}_{\Omega} $ on $A$ induced by $\Omega$.


We conclude the paper with a representative example that illustrates the scope
and geometric content of our main result. In this case, the construction can be
made explicit in local coordinates, revealing the interplay between the Poisson
structure and the induced second-order dynamics.

\begin{example}
Let $(M,\Pi)$ be a Poisson manifold and consider its cotangent Lie algebroid
$A = T^*M$ (see Example \ref{Poiss-algebroid}). In local coordinates $(U,x^1,\ldots,x^n)$ on $M$, we write
\[
\Pi = \frac{1}{2} \Pi^{ij}(x)\, \frac{\partial}{\partial x^i} \wedge \frac{\partial}{\partial x^j}.
\]
In adapted coordinates $(x^i,p_i)$ on $T^*M$, the structure functions \eqref{strucfunc} are
$\displaystyle\rho^i_j = -\Pi^{ij}(x)$ and  $\displaystyle C^k_{ij} = \frac{\partial \Pi^{ij}}{\partial x^k}$. 

Let $g$ be a Riemannian metric on $M$, with inverse $g^{ij}(x)$, and consider the Lagrangian $\displaystyle L \in C^\infty(A)$ defined by
\[
L(x,p) = \frac{1}{2} g^{ij}(x)\, p_i p_j.
\]
The Lagrangian $L$ is regular and coincides with its energy function: 
\[
E_L = L = \frac{1}{2} g^{ij}(x)\, p_i p_j
\]
Since $\Pi$ is Poisson, we can think of $\Pi\in \Gamma(\wedge^2 TM)$ as a closed section.
By Lemma \ref{lemma-hor-closed}, we have the horizontal closed section on the prolongation
$\Theta=\hat{p}^{\, *}\Pi = \frac{1}{2}\Pi^{ij}\mathcal{E}^{i}\wedge \mathcal{E}^{i}$.
Using the general construction described in Section 3, we obtaining a family of Poisson brackets on $T^*M$, locally given by
\begin{equation*}
\{  x^{i},x^{j} \}=0, \quad  \{  x^{i},p_{k} \}= -\Pi^{ir}g_{rk},  \quad    \{  p_{k},p_{l} \}= -g_{kr}N_{rs}g_{sl},
\end{equation*}
where the functions $N_{ij}$  take the form
\[
N_{ij} = \left(\Pi^{ki}(\Gamma^j_{kr}g^{rl}+\Gamma^l_{kr}g^{jr})-\Pi^{kj}(\Gamma^i_{kr}g^{rl}+\Gamma^l_{kr}g^{ir})-g^{kl}\frac{\partial \Pi^{ij}}{\partial x^k}\right)p_l + \Pi_{ij},
\]
where $\Gamma^k_{ij}$ denote the Christofel symbols of any connection satisfying $\nabla g=0$. 
Thus, the Hamiltonian semispray of $E_L=L$ is given by 
$$ \mathcal{V}(x,y)=p_j\Pi^{ji}(x)\frac{\partial}{\partial x^i}+ g_{ks}\left(N_{sr}p_r +\frac{1}{2}\Pi^{si} \left( \Gamma^r_{ij}g^{jl}+ \Gamma^l_{ij}g^{rj}\right)p_rp_l \right)\frac{\partial}{\partial p_k}. $$
As a semispray, the base dynamics of the vector field $ \mathcal{V}(x,y)$ is entirely determined by the Poisson tensor.
On the other hand, the fiber dynamic depends on the metric, the connection, and on the derivatives of the Poisson tensor. 
Specifically, the metric enters in a covariant way while the Poisson tensor contributes through its explicit variation.

If we further assume that the Poisson tensor $\Pi$ is 
$\nabla$-parallel,
then, using the identities
\[
\partial_k \Pi^{ij} = - \Gamma^i_{kl} \Pi^{lj} - \Gamma^j_{kl} \Pi^{il},
\]
one can rewrite the coefficients $N_{ij}$ in terms of the Christoffel symbols of $\nabla$. 
In this case, the resulting semispray is fully covariant. 
\end{example}

This example highlights the geometric nature of our construction 
and shows that the existence of Hamiltonian semisprays on Lie algebroids is not merely a formal extension of the tangent bundle case, 
but rather a phenomenon intrinsically tied to the interplay between Poisson and algebroid structures.

\paragraph{Acknowledgements.} 
SECIHTI supported this research under the Grant CF-2023-G-279.
 

\end{document}